\documentclass{amsart}
\usepackage[english]{babel}
\usepackage{amsmath,amssymb,amsthm}

\usepackage{enumerate}
\usepackage{color}

\newtheorem{theorem}{Theorem}

\newtheorem{corollary}[theorem]{Corollary}
\newtheorem{definition}[theorem]{Definition}

\newtheorem{quest}[theorem]{Question}

\begin{document}
\title{Derivations in the Banach ideals of $\tau$-compact operators}

\author{A. F. Ber}
\address{Department of Mathematics, Tashkent State University, Uzbekistan }
\email{ber@ucd.uz}

\author{F. A. Sukochev}
\address{School of Mathematics and Statistics, University of New South Wales, Sydney, NSW 2052, Australia }
\email{f.sukochev@unsw.edu.au}

\maketitle

\bigskip
\begin{abstract}
\noindent Let $\mathcal{M}$ be a von Neumann algebra equipped with
a faithful normal semi-finite trace $\tau$ and let $S_0(\tau)$ be
the algebra of all $\tau$-compact operators affiliated with
$\mathcal{M}$. Let $E(\tau)\subseteq S_0(\tau)$ be a symmetric
operator space (on $\mathcal{M}$) and let $\mathcal{E}$ be a
symmetrically-normed Banach ideal of $\tau$-compact operators in
$\mathcal{M}$. We study (i) derivations $\delta$ on $\mathcal{M}$
with the range in $E(\tau)$  and (ii) derivations on the Banach
algebra $\mathcal{E}$. In the first case our main results assert
that such derivations are continuous (with respect to the norm
topologies) and also inner (under some mild assumptions on
$E(\tau)$). In the second case we show that any such derivation is
necessarily inner when $\mathcal{M}$ is a type $I$ factor. As an
interesting application of our results for the case (i) we deduce
 that any derivation from $\mathcal{M}$ into an $L_p$-space,
$L_p(\mathcal{M},\tau)$, ($1<p<\infty$) associated with
$\mathcal{M}$ is inner.
\end{abstract}

\section{Introduction}
It is well known \cite[Lemma 4.1.3]{Sak} that every derivation on
a $C^*$-algebra $A$ is norm continuous. In fact, this also easily
follows from another well known fact \cite[Corollary 4.1.7]{Sak}
that every derivation on $A$ realized as a $*$-subalgebra in the
algebra $\mathcal{B}(H)$ of all bounded linear operators on a
Hilbert space $H$ is given by a reduction of an inner derivation
on a von Neumann algebra $\mathcal{M}=\overline{A}^{wo}$ (the
closure of $A$ in the weak operator topology on $\mathcal{B}(H)$).
In the special setting when $A=\mathcal{K}(H)$ (the ideal of all
compact operators on $H$) and $\mathcal{M}=\mathcal{B}(H)$, the
latter result states that for every derivation $\delta$ on $A$
there exists an operator $a\in\mathcal{B}(H)$ such that
$\delta(x)=[a,x]$ for every $x\in\mathcal{K}(H)$. The ideal
$\mathcal{K}(H)$ is a classical example of a so-called
symmetrically-normed operator ideal in $\mathcal{B}(H)$ (see
\cite{GK1, GK2, Schatten, Simon, KS}). Any such ideal
$\mathcal{E}\neq \mathcal{K}(H)$ is a Banach $*$-algebra (albeit
not a $C^*$-algebra) and two natural questions immediately
suggested by this discussion are as follows.

\begin{quest}\label{q1} Let $\mathcal{E}\subsetneqq \mathcal{K}(H)$ be a symmetrically-normed ideal of compact operators on $H$  and let $\delta\colon \mathcal{E}\to \mathcal{E}$ be a derivation on $\mathcal{E}$. Is $\delta$ continuous with respect to a symmetric (Banach) norm $\|\cdot\|_{\mathcal{E}}$ on $\mathcal{E}$?
\end{quest}

\begin{quest}\label{q2}  Does there exist an operator $a\in\mathcal{B}(H)$ such that $\delta(x)=[a,x]$ for every $x\in \mathcal{E}$?
\end{quest}

Of course, a positive answer to Question \ref{q2} implies also a
positive answer to Question \ref{q1}. However, in this paper, we
are establishing a positive answer to Question \ref{q2} via
firstly answering Question \ref{q1} in positive. In fact, we
consider a much more general situation when $\mathcal{B}(H)$ is
replaced with an arbitrary von Neumann algebra $\mathcal{M}$ and
an arbitrary symmetrically-normed ideal $\mathcal{E}$ in
$\mathcal{M}$ is replaced with a Banach $\mathcal{M}$-bimodules of
$\tau$-compact operators affiliated with $\mathcal{M}$ (here
$\tau$ is a faithful normal semifinite trace on  $\mathcal{M}$).
In particular, we study in depth the situation when such a
bimodule is a symmetric space of $\tau$-measurable operators  (see
\cite{Ovchinnikov2, SC, KS}) which is an extension of the
classical study of symmetrically-normed ideals of compact
operators. In the setting of $\tau$-measurable operators,
we also contribute to some open questions in the literature as follows.
Let us consider the classical example of the algebras $S(0,1)$ (respectively, $L_\infty(0,1)$) of
of all (classes) of Lebesgue (respectively, essentially bounded) measurable functions on the interval
$[0,1]$. Denote $D(0,1)$ the subset of $S(0,1)$, consisting of
all classes containing a function having a finite derivative (that
is $\frac {d}{dt}$)almost everywhere. Obviously, $D(0,1)$ is a
subalgebra in $S(0,1)$ and we have a correctly defined linear
mapping $\delta : D(0,1) \mapsto S(0,1)$, given by
$\delta(f)(t)=f'(t)$ for all $t \in (0,1)$, for which the finite
derivative $f'$ of $f$ exists. It is easy to verify
that $\delta$ is a derivation, which we still denote $\frac
{d}{dt}$. Motivated by this observation, Sh.Ayupov (see \cite{A1,A2}) suggested to study derivations in the setting of algebras of
of $\tau$-measurable operators and came up with the following explicit question.

\begin{quest}\label{q3} Is any derivation $\delta$ on the algebra of all $\tau$-
measurable operators affiliated with an arbitrary von Neumann
algebra $\mathcal{M}$ necessarily inner?
\end{quest}

In the setting when $\mathcal{M}$ is commutative, e.g. when
$\mathcal{M}=L_\infty(0,1)$, in view of the fact that every inner
derivation on $S(0,1)$ vanishes, this question asks whether
any non-trivial derivation on such an algebra exists.
This question, in the setting of the algebra $S(0,1)$, was settled in the affirmative in \cite{BCS1} where it is established that there exists a multitude of non-trivial derivations from $L_\infty(0,1)$ into $S(0,1)$ extending $\frac{d}{dt}$. Our contribution in this paper shows that the range of every such non-trivial derivation $\delta$ on $L_\infty(0,1)$ is not contained in any Banach $L_\infty(0,1)$-bimodule in the algebra  $S(0,1)$ (see Corollary \ref{bcom}).
\section{Preliminaries}

For details on the von Neumann algebra theory, the reader is referred
to e.g. \cite{Dix, Sak, Tak, SZ, KR2}.


\subsection{The algebra of $\tau$-measurable operators}
Let $\mathcal{M}$ be a von Neumann algebra on a Hilbert space $H$
equipped with a semi-finite normal faithful trace $\tau$. The set
of all self-adjoint projections in $\mathcal{M}$ is denoted by
$\mathcal{P}\left(\mathcal{M}\right)$. The commutant of $\mathcal{M}$ is denoted by $\mathcal{M}^{\prime}$.
A linear operator $x:\mathfrak{D}\left( x\right) \rightarrow
H $, where the domain $\mathfrak{D}\left( x\right) $ of $x$ is a linear
subspace of $H$, is said to be \emph{affiliated} with $\mathcal{M}$ (notation $x\eta \mathcal{M}$) if $yx\subseteq
xy$ for all $y\in \mathcal{M}^{\prime }$. For every self-adjoint operator $x\eta\mathcal{M}$, we have $e^x\subset \mathcal{P}\left(\mathcal{M}\right)$ (here $e^{ x }$ denotes the spectral measure of $x$).
Furthermore, if $x$ is a closed operator in $H$ with the polar decomposition $x = u|x|$
and $x\eta \mathcal{M}$, then $u\in\mathcal{M}$ and $|x|\eta \mathcal{M}$





Recall that a closed and densely
defined linear operator $x$ on ${H}$ is called $\tau $\emph{
-measurable} if $x$ is affiliated with $\mathcal{M}$  and there exists $0<\lambda\in \mathbb{R}$ such that $\tau \left(
e^{\left\vert x\right\vert }\left( \lambda,\infty \right) \right) <\infty $. The collection of all $\tau $-measurable operators is denoted by
$S\left( \tau \right) $. With respect to the strong sum and product, $S\left( \tau \right) $ is a $\ast $-algebra.

\subsection{Generalized singular value function and $\tau$-compact operators} For any $x\in S\left( \tau \right) $ its \emph{generalized singular value
function} $\mu \left( x\right) :\left[ 0,\infty \right) \rightarrow \left[
0,\infty \right] $ is defined by setting
\begin{equation*}
\mu \left( t;x\right) =\inf \left\{ s\geq 0:\tau \left( e^{\left\vert
x\right\vert }\left( s,\infty \right) \right) \leq t\right\} ,\ \ \ t\geq 0.
\end{equation*}
Note that $\mu \left( t;x\right) <\infty $ for all $t>0$ and that is
decreasing and right continuous. The set $S_{0}\left( \tau \right) $ of all $\tau $\emph{-compact operators} is defined by
\begin{equation*}
S_{0}\left( \tau \right) =\left\{ x\in S\left( \tau \right)
:\lim\nolimits_{t\rightarrow \infty }\mu \left( t;x\right) =0\right\} .
\end{equation*}
The set $S_{0}\left( \tau \right) $ is a $\ast $-closed two sided ideal in $S\left( \tau \right) $, which is closed for the measure topology \cite{SW}.
Consider $\mathcal{M}=L^\infty([0,\infty))$ as an Abelian von Neumann algebra acting via multiplication on the Hilbert space $H=L^2(0,\infty)$, with the trace given by integration with respect to Lebesgue measure $m$. It is easy to see that the set of all $\tau$-measurable operators affiliated with $\mathcal{M}$ coincides with the set $S(0,\infty)$ of all measurable function on $[0,\infty)$ which are bounded expect on a set of finite measure, and that the generalized singular value function $\mu(x)$, $x\in S(0,\infty)$ is precisely the non-increasing right-continuous rearrangement of $x$ given by $$x^*(t)=\inf\bigl\{ s\geqslant 0: m\left(\{|x|\geqslant s\}\right)\leqslant t\bigl\}$$(see e.g. \cite{KPS, DPS}). In this setting, the algebra $S_{0}\left( \tau \right) $ coincides with the set $S_0(0,\infty)$ of all $f\in S(0,\infty)$ such that $f^*(\infty):=\lim_{t\to \infty}f^*(t)=0$.

If $\mathcal{M}=\mathcal{B}(H)$ (respectively, $l_\infty(\mathbb{N}))$ and $\tau$ is the standard trace $Tr$ (respectively, the counting measure on $\mathbb{N})$, then it is not difficult to see that $S(\tau)=\mathcal{M}$. In the first case, for $x\in S(\tau)$ by setting $$s(x):=\{s_n(x)\}_{n=1}^\infty,\quad s_n(x)=\mu(t;x), t\in(n-1,n], n=0,1,2,\dots,$$  we recover  the sequence $s(x)$ of singular values of $x$ (see e.g. \cite{GK1, GK2, Simon}). In this case, $S_{0}\left( \tau \right) $ is the set $\mathcal{K}(H)$ of all compact operators on $H$. In the second case,  when $x$ is a sequence we denote by $x^*$ the usual decreasing rearrangement of the sequence $|x|$ and identify $\mu(x)$ with $x^*$. The algebra $S_{0}\left( \tau \right) $ then coincides with the set $c_0$ of all vanishing sequences. For more details on the theory of $\tau$-compact operators we refer the reader to \cite{DPS, Ovchinnikov3, Sonis, SW}.

\subsection{Banach $\mathcal{M}$-bimodules}
As before, we assume that $\left( \mathcal{M},\tau \right) $ is a
semi-finite von Neumann algebra. We recall the following definition.

\begin{definition}\label{bimodule}
A linear subspace $\mathbb E$ of $S(\tau)$, equipped with
a norm $\| \cdot \| _{\mathbb E}$, is called a \emph{Banach $\mathcal{M}$-bimodule (of }$\tau $\emph{-measurable operators)} if

\begin{enumerate}
\item[(i)] if $uxv\in \mathbb E$ and $\left\Vert uxv\right\Vert _{\mathbb E}\leq \left\Vert
u\right\Vert _{\mathcal{B}\left( {H}\right) }\left\Vert v\right\Vert _{\mathcal{B}\left(
{H}\right) }\left\Vert x\right\Vert _{\mathbb E}$ whenever $x\in \mathbb E$ and $
u,v\in \mathcal{M}$;

\item[(ii)] $\left( \mathbb E,\left\Vert \cdot \right\Vert _{\mathbb E}\right) $ is a Banach
space.
\end{enumerate}
\end{definition}

It should be observed that any Banach $\mathcal{M}$-bimodule $\mathbb E$ is
$\ast $-closed and that $x\in S\left( \tau \right) $, $y\in \mathbb E$ and
$\left\vert x\right\vert \leq \left\vert y\right\vert $ imply that
$x\in \mathbb E$ and $\left\Vert x\right\Vert _{\mathbb E}\leq \left\Vert
y\right\Vert _{\mathbb E}$ (see \cite{DPS}).


We now recall two well known special cases of Definition \ref{bimodule}.

\begin{definition}\label{opideal}
An ideal $\mathcal E$ in the von Neumann algebra
$\mathcal M$ equipped with a Banach norm $\|\cdot\|_{\mathcal E}$ is said
to be \emph{a symmetrically-normed operator ideal} if
 $$\| ASB\|_{\mathcal E}\leq \| A\| \:\| S\|_{\mathcal E}\| B\|\mbox{ for all }\
S\in \mathcal E,\ A,B\in \mathcal M.$$
\end{definition}


An important and interesting class of Banach $\mathcal{M}$-bimodules is
that of symmetric spaces of $\tau $-measurable operators on $\mathcal{M}$.

\begin{definition}\label{symopspace} Let ${E}(\tau)$  be a linear subspace in $S(\tau)$ equipped with a norm $\|\cdot\|_{{E}(\tau)}$. We say that ${E}(\tau)$ is a \emph{symmetric operator space} (on $\mathcal{M}$, or in $S(\tau))$ if for any $x\in{E}(\tau)$ and every $y\in S(\tau)$ such that $\mu(y)\leqslant\mu(x)$, we have $y\in{E}(\tau)$ and $\|y\|_{{E}(\tau)}\leqslant\|x\|_{{E}(\tau)}$.
\end{definition}
%
%
In the special case, when $\mathcal{M}=\mathcal{B}(H)$ and $\tau$ is a standard trace $Tr$, the class of symmetric operator spaces introduced in Definition \ref{symopspace} coincides with that of  symmetrically-normed operator ideals given in Definition \ref{opideal}.

Let $L_1(0,1)$ (respectively, $L_1(0,\infty))$ be the space of all integrable functions on $(0,1)$ (respectively, on $(0,\infty)$), $L_\infty(0,1)$ (respectively, $L_\infty(0,\infty))$ be the space of all essentially bounded measurable functions on $(0,1)$ (respectively, on $(0,\infty)$), $l_1$ and $l_\infty$ are the classical spaces of all absolutely summable and bounded sequences respectively.

Let $E$ be a Banach space of real-valued Lebesgue measurable function either on $(0,1)$ or $(0,\infty)$ (with identification $m$-a.e.) or on $\mathbb{N}$. The space $E$ is said to be \textit{absolutely solid} if $x\in E$ and $|y|\leqslant|x|,\ y\in S(0,\infty)$ implies that $y\in E$ and $\|y\|_E\leqslant\|x\|_E$.

The absolutely solid space $E\subset S$ is said to be \textit{symmetric} if for every $x\in E$ and every $y$ the assumption $y^*=x^*$ implies that $y\in E$ and $\|y\|_E=\|x\|_E$ (see e.g. \cite{KPS}).
If $E=E(0,1)$ is a symmetric space on $(0,1)$, then $$L_\infty(0,1)\subseteq E\subseteq L_1(0,1).$$
If $E=E(0,\infty)$ is a symmetric space on $(0,\infty)$, then $$L_1(0,\infty)\cap L_\infty(0,\infty)\subseteq E\subseteq L_1(0,\infty)+L_\infty(0,\infty).$$
If $E=E(\mathbb{N})$ is a symmetric space on $\mathbb{N}$, then $$l_1\subseteq E\subseteq l_\infty.$$
There exists a strong connection between symmetric function and operator spaces, which extends classical Calkin correspondence between two-sided ideals of compact operators and Calkin subspaces in $l_\infty$ (see \cite{Calkin}, or B.~Simon's book, \cite[Theorem 2.5]{Simon}). For brevity, we assume that $\tau(\mathbf{1})=1$ or else $\tau(\mathbf{1})=\infty$ (here,  $\mathbf{1}$ is the identity in $\mathcal{M}$).


\begin{theorem}\cite[Theorem 8.11]{KS}\label{ks} (i) For every symmetrically-normed ideal $E$ in $l_{\infty},$ the set
$$\mathcal{E}:=\{x\in B(H):\ s(x)\in E\},\quad \|x\|_{\mathcal{E}}:=\|s(x)\|_E$$
is a symmetrically-normed ideal in $B(H).$

(ii) For every semifinite von Neumann algebra $\mathcal{M}$ with a faithful normal semifinite trace $\tau$ and every symmetric space $E=E(0,\tau(\mathbf{1})))$, the set
$$E(\tau):=\{x\in S(\tau):\mu(x)\in E\}, \quad\|x\|_{E(\tau)}:=\|\mu(x)\|_E$$
is a symmetric operator space on $\mathcal{M}$.
\end{theorem}

If $E=L_p, 1\leqslant p<\infty$, then $(E(\tau), \|\cdot\|_{E(\tau)})$ coincides with the classical non-commutative $L_p$-space associated with the algebra $(\mathcal{M},\tau)$. If $\mathcal{M}$ is a semi-finite atomless von Neumann algebra, then the converse results also holds \cite{SC}. That is, if $F$ is a symmetric operator space on $\mathcal{M}$, then
$$E(0,\infty):=\{f\in S(0,\infty): f^*=\mu(x) \mbox{ for some } x\in F\}, \quad\|f\|_E:=\|x\|_F$$
is a symmetric function space on $(0,\tau(\mathbf{1}))$. It is obvious that $F=E(\mathcal{M},\tau)$.

Similarly, every symmetric operator ideal $\mathcal{E}$ in $\mathcal{M}$ defines a unique symmetric sequences spaces $E=E(\mathbb{N})$ by setting
$$E:=\{a=(a_n)_{n=1}^\infty\in l_\infty: a^*=s(x) \mbox{ for some }x\in\mathcal{E}\},\quad \|a\|_E:=\|x\|_\mathcal{E}.$$

Let $A$ be a complex algebra and $\mathbb{E}$ let be a bimodule on $A$. Recall that a \textit{derivation} from $A$ into $\mathbb{E}$ is a linear map
$\delta :A\rightarrow \mathbb{E}$ such that
\begin{equation*}
\delta \left( xy\right) =\delta \left( x\right) y+x\delta \left(
y\right) ,\ \ \ x,y\in A.
\end{equation*}
If, in addition, $A$ is a $*$-algebra, $\mathbb{E}$ is a $*$-bimodule and the derivation $\delta$ satisfies $\delta(x^*)=(\delta(x))^*$ for all $x\in A$, then it is called a $*$-derivation.
For every derivation $\delta:A\to \mathbb{E}$ of a $*$-algebra $A$ into a $*$-bimodule $\mathbb{E}$, we define mappings  $\delta_\mathrm{Re}(x):=\frac{\delta(x)+\delta(x^*)^*}{2}$ and $\delta_\mathrm{Im}(x):=\frac{\delta(x)-\delta(x^*)^*}{2i}, x\in A$. It is easy to see that
$\delta_\mathrm{Re}$ and $\delta_\mathrm{Im}$ are $*$-derivations on $A$, moreover $\delta=\delta_\mathrm{Re}+i\delta_\mathrm{Im}$. If $a\in A$, then the mapping $\delta _{a}:A\rightarrow A$, given by
$\delta _{a}\left( x\right) =\left[ a,x\right] $, $x\in A$, is a
derivation. A derivation of this form is called \textit{inner}. Similarly, a derivation $\delta:A\to \mathbb{E}$ such that $\delta \left( x\right) =\left[ a,x\right] $, for all $x\in A$ and some $a\in \mathbb{E}$ is inner.


\section{Main Results}
Throughout this section $\mathcal{M}$ is a semifinite von Neumann
algebra equipped with a faithful normal semifinite trace $\tau$.

The following theorem is our first main result. It yields a
positive answer to Question \ref{q1} (see also Corollary \ref{SS}
below) in the general setting of $\tau$-compact ideals in von
Neumann algebras. It is convenient to isolate the key algebraic
property of such algebras which underlies our proof. Let $p,q\in
\mathcal{P}\left(\mathcal{M}\right)$. The projections $p$ and $q$
are said to be equivalent, if there exists a partial isometry
$v\in\mathcal{M}$, such that $v^*v=p,\ vv^*=q$. In this case, we
write $p\sim q$.

\begin{theorem}\label{th1}
Let  $(\mathcal{E},\|\cdot\|_{\mathcal{E}})$ be an ideal of
$\tau$-compact operators in $\mathcal{M}$ equipped with a Banach
norm $\|\cdot\|_{\mathcal{E}}$ so that
$(\mathcal{E},\|\cdot\|_{\mathcal{E}})$ is a Banach
$\mathcal{M}$-bimodule and let $\delta$ be a derivation on
$\mathcal{E}$. Then $\delta$ is a continuous mapping on
$(\mathcal{E}, \|\cdot\|_{\mathcal{E}})$.
\end{theorem}
\begin{proof}
Without loss of generality, we may assume that $\delta$ is a
$*$-derivation (see the end of the preceding section). Since
$(\mathcal{E},\|\cdot\|_\mathcal{E})$ is a Banach space, it is
sufficient to prove that the graph of $\delta$ is closed. Suppose
a contrary. Then there exist a sequence
$\{a_n\}_{n=1}^\infty\subset \mathcal{E}$ and an element $0\neq
a\in \mathcal{E}$ such that $a=a^*$,
$\|a_n\|_\mathcal{E}\rightarrow 0$ and
$\|\delta(a_n)-a\|_\mathcal{E}\rightarrow 0$ as $n\to \infty$.

Let $a=a_+-a_-$ be an orthogonal decomposition of $a$, that is $
 a_+,a_-\in \mathcal{E}$, $a_+,a_-\geqslant 0$, and $a_+a_-=0$. Without loss of generality,
we may assume that $a_+\neq 0$, otherwise we consider the sequence
$\{-a_n\}_{n=1}^\infty$. Since $a\in \mathcal{E}$, there exists a
projection $p\in\mathcal{M}$ such that $pap\geqslant\lambda p$ for
some $\lambda>0$. Replacing $a_n$ with $\frac{a_n}{\lambda}$ we
may assume $pap\geqslant p$. Hence, for some operator $c\in
\mathcal{M}$, we have $p=c^*papc\in\mathcal{E}$.

There are two possible cases: (i) there exists an atom $0\neq
q\in\mathcal{P}(\mathcal{M})$ such that $q\leq p$ and (ii) the
logic $\mathcal{P}(\mathcal{M})$ does not contain atoms $q\neq 0$
such that $q\leq p$. In the case (i), we have $q\in\mathcal{E}$
and $q\leq qaq$. Since $q$ is an atom, it follows $qa_nq=\lambda_n
q$, $\lambda_n\in\mathbb{C}$ and we immediately deduce that
$\lim_{n\rightarrow\infty} \lambda_n =0$ from the assumption
$\|a_n\|_\mathcal{E}\rightarrow 0$. Observing that
$$
\|\delta(qa_nq)-q\delta(a_n)q\|_\mathcal{E}\leq
2\|\delta(q)\|_{\mathcal{M}}\|a_n\|_\mathcal{E}\to 0,\ {\mbox as}\
n\to \infty.
$$
we obtain a contradiction with the assumption $q\neq 0$ as follows
$$q\leq
qaq=\|\cdot\|_{\mathcal{E}}-\lim_{n\rightarrow\infty}
\delta(qa_nq)=\|\cdot\|_{\mathcal{E}}-\lim_{n\rightarrow\infty}
\delta(\lambda_n q)=\delta(q)\lim_{n\rightarrow\infty} \lambda_n
=0.$$

In the case (ii), there exists a pairwise orthogonal sequence
$\{e_n\}_{n=1}^\infty\subset \mathcal{P}(\mathcal{M})$ such that
$0\neq e_n\leq p$, for all $n\ge 1$. Clearly, we have
$\{e_n\}_{n=1}^\infty\subset\mathcal{E}$ and $e_nae_n\geq e_n$ for
any $n\in\mathbb{N}$. Let $\{m_n\}_{n=1}^\infty$ be any sequence
of positive integers such that
$$
m_n>(2n+1)/\|e_n\|_{\mathcal{E}},\quad n\ge 1.
$$
Passing to a subsequence if necessary, we may assume without loss of generality that
$$\|a_n\|_{\mathcal{E}}<m_n^{-1}2^{-n},\
\|\delta(a_n)-a\|_{\mathcal{E}}<m_n^{-1}$$ and that
$$\|a_n\|_{\mathcal{E}}<n m_n^{-1}\|\delta(e_n)\|_{\mathcal{M}}^{-1}$$ whenever $n\ge 1$ is such that $\delta(e_n)\neq 0$. Let us define an element
$$c:=\sum_{n=1}^\infty m_ne_na_ne_n\in\mathcal{E}$$
where the series converges in the norm $\|\cdot\|_{\mathcal{E}}$,
since we have $\|m_ne_na_ne_n\|_{\mathcal{E}}<2^{-n}$. We intend
to obtain a contradiction by showing that the norm
$\|\delta(c)\|_{\mathcal{E}}$ is larger than any positive integer
$n$. Indeed, fixing such $n\ge 1$, we have
$\|\delta(c)\|_{\mathcal{E}}\ge \|e_n\delta(c)e_n\|_{\mathcal{E}}$
(due to the fact that $\mathcal{E}$ is an $\mathcal{M}$-bimodule)
and
\begin{align*}
\|e_n\delta(c)e_n\|_{\mathcal{E}}&=
\|\delta(e_nc)e_n-\delta(e_n)ce_n\|_{\mathcal{E}}=
m_n\|\delta(e_na_ne_n)e_n-\delta(e_n)e_na_ne_n\|_{\mathcal{E}}\\
&= m_n\|e_n\delta(a_ne_n)e_n\|_{\mathcal{E}}=
m_n\|e_n\delta(a_n)e_n+e_na_n\delta(e_n)e_n\|_{\mathcal{E}}\\
&\geq
m_n\|e_n(\delta(a_n)-a)e_n+e_nae_n\|_{\mathcal{E}}-m_n\|e_na_n\delta(e_n)e_n\|_{\mathcal{E}}\\
&\geq
m_n(\|e_nae_n\|_{\mathcal{E}}-\|e_n(\delta(a_n)-a)e_n\|_{\mathcal{E}})-
m_n\|a_n\|_{\mathcal{E}}\|\delta(e_n)\|_{\mathcal{M}}\\ & \geq
m_n(\|e_nae_n\|_{\mathcal{E}}-\|\delta(a_n)-a\|_{\mathcal{E}})-n\\
&
> m_n\|e_nae_n\|_{\mathcal{E}}-1-n\geq
m_n\|e_n\|_{\mathcal{E}}-1-n>n.
\end{align*}
This shows that $\delta$ is a continuous mapping on
$(\mathcal{E},\|\cdot\|_{\mathcal{E}})$.
\end{proof}
\begin{corollary}\label{SS} (i). Let $(E, \|\cdot\|_E)$ be a symmetric space of measurable functions, such that $E\subseteq S_0(0,\infty)\cap L_\infty(0,\infty)$ and let $\mathcal{E}=E(\tau)$. Then any derivation $\delta:\mathcal{E}\to \mathcal{E}$ is continuous with respect to the norm on $\mathcal{E}$ given by $\|x\|_{\mathcal{E}}=\|\mu(x)\|_E$.

(ii). Let $E$ be a symmetric sequence space, $\mathcal{M}=\mathcal{B}(H)$  and $\mathcal{E}$ be a symmetrically-normed ideal of compact operators on $H$ generated by $E$. Then any derivation $\delta:\mathcal{E}\to \mathcal{E}$ is continuous with respect to the norm on $\|\cdot\|_\mathcal{E}$.

\end{corollary}
\begin{proof} Theorem \ref{th1} and Theorem \ref{ks}.
\end{proof}
In the case when $\mathcal{M}$ is commutative, the result of Corollary \ref{SS}(i) can be further strengthened.
\begin{theorem}\label{th1_1}
Let $\mathcal{M}$ be a commutative von Neumann algebra and let
$(\mathcal{E},\|\cdot\|_{\mathcal{E}})$ be a Banach ideal of
$\tau$-compact operators in $\mathcal{M}$. Any derivation $\delta$
on $\mathcal{E}$ vanishes.
\end{theorem}
\begin{proof}
First of all, we observe that $\delta(p)=0$ for any $p\in\mathcal{P}(\mathcal{M})$. Indeed, this easily follows from applying $\delta$ to $p^2$ and $p^3$ and comparing the results.
 Suppose that  $\delta(a)\neq 0$ for some $a\in\mathcal{M}$. Without loss of generality, we may assume $\delta^*=\delta,\ a^*=a$ and so $\delta(a)\in \mathcal{E}\cap \mathcal{M}_h$. Further, again without loss of generality we may assume $\delta(a)_+\neq 0$ and so there exists a projection $0\neq p\in\mathcal{P}(\mathcal{M})$ such that $\delta(a)p\geq p$.
 Obviously, $p\in\mathcal{E}$ and since $\delta(ap)=\delta(a)p\geq p$, we may assume that $ap=a$ (in particular, $\delta(a)=\delta(ap)=\delta(a)p$). Our next step is to show that we may also assume that $a\geq0$. Let $e_-=s(a_-)\leq s(a)\leq p,\ e_+=s(a_+)\leq s(a)\leq p$. We have either $e_+>0$, or else $e_->0$. In the first case, we set $a=ae_+,\ p=pe_+$, and in the second case $a=-ae_-,\ \delta=-\delta,\ p=pe_-$.
 We have arrived to the following setting: $\mathcal{E}\ni a\geq 0,\ ap=a,\ \delta(a)\geq p>0$. There are two possible cases: (i) there exists an atom $0\neq q\in\mathcal{P}(\mathcal{M})$ such that $q\leq p$ and
(ii) the logic $\mathcal{P}(\mathcal{M})$ does not contain atoms
$q\neq 0$ such that $q\leq p$. In the case (i), we have
$q\in\mathcal{E}$ and $aq=\lambda q$, where $\lambda
\in\mathbb{R}$. In this case, we obtain a contradiction as
follows:
$$0=\lambda\delta(q)=\delta(\lambda q)=\delta(aq)=\delta(a)q\geq q>0.$$

In the case (ii), there exists a pairwise orthogonal sequence
$\{e_n\}_{n=1}^\infty\subset \mathcal{P}(\mathcal{M})$ such that
$0\neq e_n\leq p$, for all $n\ge 1$. Obviously,
$e_n\in\mathcal{E}$, $n\ge 1$. Fix $0\leq x\in\mathcal{E}$ and
define
$$[x]=\sum_{0<k\leq \|x\|_{\mathcal{M}}}ke^x[k,k+1),$$ where $e^x$ is the spectral measure of the operator $x$.
Clearly, $0\leq [x]\leq x$ and therefore $[x]\in\mathcal{E}$.
Since, $[x]$ is a finite combination of projections from
$\mathcal{P}(\mathcal{M})$ we have  $\delta([x])=0$. Thus,
$\{x\}:=x-[x]\in\mathcal{E},\ 0\leq\{x\}\leq \mathbf{1}$ and
$\delta(\{x\})=\delta(x)$.

Let $\{m_n\}_{n=1}^\infty$ be any sequence of positive integers
such that $m_n\|e_n\|_{\mathcal{E}}>n$. Let us define an element
$$c=(so)-\sum_{n=1}^\infty \{m_n a\} e_n$$ (the series converges in the strong operator topology and $\|c\|_{\mathcal{M}}\leq 1$). We have $c=cp\in\mathcal{E}$ and therefore
\begin{align*}
\|\delta(c)\|_{\mathcal{E}}=\|\delta(c)\|_{\mathcal{E}}\|e_n\|_{\mathcal{M}}&
\geq\|\delta(c)e_n\|_{\mathcal{E}}=\|\delta(ce_n)\|_{\mathcal{E}}\\
& =
\|\delta(\{m_na\}e_n)\|_{\mathcal{E}}=\|\delta(\{m_na\})e_n\|_{\mathcal{E}}=\|\delta(m_na)e_n\|_{\mathcal{E}}\\
& =m_n\|\delta(a)e_n\|_{\mathcal{E}}\geq
m_n\|e_n\|_{\mathcal{E}}>n
 \end{align*}
for any $n\in\mathbb{N}$. This contradiction proves that $\delta=0$.
\end{proof}

The following theorem is our second main result. It answers
Question \ref{q2} in the affirmative.

\begin{theorem}\label{th41}
Let $\mathcal{E}$  be a symmetrically-normed ideal in
$\mathcal{B}({H})$. Then for any derivation $\delta:
\mathcal{E}\rightarrow \mathcal{E}$ there exists an operator
$d\in\mathcal{B}({H})$, such that
$\delta(\cdot)=\delta_d(\cdot):=[d,\cdot]$. In addition, we have
$\|d\|_{\mathcal{B}({H})}\leq\|\delta\|_{\mathcal{E}\rightarrow
\mathcal{E}}$.
\end{theorem}
\begin{proof} By Theorem \ref{ks}, every symmetrically-normed ideal $\mathcal{E}$  is defined by a symmetric sequence space $E$. Every such space $E$ is continuously embedded into $l_\infty$ (see e.g. \cite{KPS}), and this immediately implies that the embedding $\mathcal{E}\subset \mathcal{B}(H)$ is continuous. Without loss of generality, we may assume that the embedding constant is equal to $1$.
Fix an arbitrary vector $\xi_0\in{H}$ such that $\|\xi_0\|_H=1$.
Let $p_0:=(\xi_0,.)\xi_0$, that is $p_0$ is an orthogonal
projection onto one-dimensional subspace spanned by $\xi_0$.
Therefore $p_0\in\mathcal{E}$.

Let $a\in\mathcal{E}$ be such that $a\xi_0=0$. Then $ap_0=0$ and
therefore $\delta(ap_0)=0$. In particular, if operators $b_1,
b_2\in \mathcal{E}$ are such that $b_1\xi_0=b_2\xi_0$, then
$(b_1-b_2)\xi_0=0$ and therefore
$\delta(b_1p_0)-\delta(b_2p_0)=\delta((b_1-b_2)p_0)=0$. Let us now
set
$$d(b\xi_0):=\delta(bp_0)\xi_0,\quad b\in \mathcal{E}.$$
The preceding comment shows that $d$ is correctly defined. Since
any element $\eta\in{H}$ may be written as $\eta=b\xi_0$, for some
$b\in\mathcal{E}$ (it is sufficient to set
$b(\cdot):=(\cdot,\xi_0)\eta$), we see that the operator $d$ is a
well defined linear operator on  ${H}$.

Next, observe that $|bp_0|^2=p_0b^*bp_0=(b^*b\xi_0,\xi_0)p_0$, and
therefore $|bp_0|=\|b\xi_0\|_Hp_0$ yielding
$\|bp_0\|_{\mathcal{E}}=
\||bp_0|\|_{\mathcal{E}}=\|\|b\xi_0\|_Hp_0\|_{\mathcal{E}}=
\|b\xi_0\|_H\|p_0\|_{\mathcal{E}}$ for any $b\in \mathcal{E}$. To
verify that $d$ is a bounded operator on $H$, we firstly recall
that by Theorem \ref{th1} we have
$\|\delta\|:=\|\delta\|_{\mathcal{E}\to \mathcal{E} }<\infty$ and
then write
\begin{eqnarray*}
\|d(b\xi_0)\|_H=\|\delta(bp_0)\xi_0\|_H=\|\delta(bp_0)p_0\|_{\mathcal{E}}\|p_0\|_{\mathcal{E}}^{-1}\leq
\|\delta\|\cdot\|bp_0\|_{\mathcal{E}}\|p_0\|_{\mathcal{E}}^{-1}=\|\delta\|\cdot\|b\xi_0\|_H.
\end{eqnarray*}

We conclude $d\in \mathcal{B}(H)$ and
$\|d\|_{\mathcal{B}(H)}\leq\|\delta\|$.

Finally,
$$[d,x](b\xi_0)=dx(b\xi_0)-xd(b\xi_0)=d(xb\xi_0)-xd(b\xi_0)=\delta(xbp_0)\xi_0-x\delta(bp_0)\xi_0=\delta(x)b\xi_0,$$
for any $x,b\in\mathcal{E}$, which means
$\delta(\cdot)=[d,\cdot]$.
\end{proof}
The following theorem is a complement to our first main result. Its proof is very similar to that of Theorem \ref{th1}. We provide an outline of the proof, indicating the differences.

\begin{theorem}\label{th2}Suppose that $\mathbb{E}$ is a Banach bimodule of
$\tau$-compact operators. Let $\delta\colon\mathcal{M}\to
\mathbb{E}$ be a derivation. Then $\delta$ is a continuous mapping
from $(\mathcal{M},\|\cdot\|_\mathcal{M})$ to
$(\mathbb{E},\|\cdot\|_{\mathbb{E}})$.
\end{theorem}
\begin{proof} Without loss of generality, we may assume that
$\delta$ is a $*$-derivation and, as in the proof of Theorem
\ref{th1},  that the graph of $\delta$ is not closed, in
particular, that there exists a sequence
$\{a_n\}_{n=1}^\infty\subset\mathcal{M}$ such that
$\lim_{n\rightarrow\infty}\|a_n\|_{\mathcal{M}}=0,\
\lim_{n\rightarrow\infty}\|\delta(a_n)-a\|_{\mathbb{E}}=0$ for
some $a=a^*\neq 0, a\in \mathbb{E}$ with $a_+\neq 0$. Since
$a_+\in S_0(\tau)$ there exists a projection $p\in\mathcal{M}$
such that $\tau(p)<\infty$ and $\mathcal{M}\ni pap\geqslant\lambda
p$ for some $\lambda>0$. Replacing $a_n$ with
$\frac{a_n}{\lambda}$ we may assume $pap\geqslant p\in
\mathbb{E}$.

As in the proof of Theorem \ref{th1}  there are two cases (i) and (ii). In the case (i) the proof is a verbatim repetition of the argument from Theorem \ref{th1}.


In the case (ii) we shall use the same sequence
$\{e_n\}_{n=1}^\infty$ as in Theorem \ref{th1}. We have
$\{e_n\}_{n=1}^\infty\subset \mathbb{E}$ and $e_nae_n\geq e_n$ for
any $n\in\mathbb{N}$.  Let $\{m_n\}_{n=1}^\infty$ be any sequence
of positive integers as in Theorem \ref{th1}. Passing to a
subsequence if necessary, we may assume without loss of generality
that
$$\|a_n\|_{\mathcal{M}}<m_n^{-1},\
\|\delta(a_n)-a\|_{\mathbb{E}}<m_n^{-1}$$ and that
$$\|a_n\|_{\mathcal{M}}<n m_n^{-1}\|\delta(e_n)\|_{\mathbb{E}}^{-1}$$ whenever $n\ge 1$ is such that $\delta(e_n)\neq 0$.
Let us define an element
$$c:=(so)-\sum_{n=1}^\infty m_ne_na_ne_n\in\mathcal{M}$$
(the series above is composed of pairwise orthogonal elements whose operator norm does not exceed $1$, and therefore the series converges in the strong operator topology, in particular,
$\|c\|_{\mathcal{M}}\leq 1$). As in Theorem \ref{th1}, we obtain a contradiction by showing that the norm $\|\delta(c)\|_{\mathbb{E}}$ is larger than any positive integer $n$ as follows:
\begin{align*}
\|\delta(c)\|_{\mathbb{E}}&
\geq\|e_n\delta(c)e_n\|_{\mathbb{E}}\\& =
\|\delta(e_nc)e_n-\delta(e_n)ce_n\|_{\mathbb{E}}=m_n\|\delta(e_na_ne_n)e_n-\delta(e_n)e_na_ne_n\|_{\mathbb{E}}\\
&=
m_n\|e_n\delta(a_ne_n)e_n\|_{\mathbb{E}}=m_n\|e_n\delta(a_n)e_n+e_na_n\delta(e_n)e_n\|_{\mathbb{E}}\\&\geq
m_n\|e_n\delta(a_n)e_n\|_{\mathbb{E}}-m_n\|e_na_n\delta(e_n)e_n\|_{\mathbb{E}}\\
& \geq
m_n\|e_n(\delta(a_n)-a)e_n+e_nae_n\|_{\mathbb{E}}-m_n\|a_n\|_{\mathcal{M}}\|\delta(e_n)\|_{\mathbb{E}}\\
& \geq
m_n(\|e_nae_n\|_{\mathbb{E}}-\|e_n(\delta(a_n)-a)e_n\|_{\mathbb{E}})-m_n\|a_n\|_{\mathcal{M}}\|\delta(e_n)\|_{\mathbb{E}}\\
&\geq
m_n(\|e_nae_n\|_{\mathbb{E}}-\|\delta(a_n)-a\|_{\mathbb{E}})-m_n\|a_n\|_{\mathcal{M}}\|\delta(e_n)\|_{\mathbb{E}}\\
&> m_n\|e_nae_n\|_{\mathbb{E}}-1-n\geq
m_n\|e_n\|_{\mathbb{E}}-1-n>n.
\end{align*}
\end{proof}

\begin{corollary}\label{SS2}
 Let $(E, \|\cdot\|_E)$ be a symmetric space of measurable functions, such that $E\subseteq S_0(0,\infty)$. Then any derivation $\delta:\mathcal{M}\to E(\tau)$ is continuous from $(\mathcal{M},\|\cdot\|_\mathcal{M})$ into $(E(\tau),\|\cdot\|_{E(\tau)})$.
\end{corollary}
\begin{proof} Theorem \ref{th2} and Theorem \ref{ks}.
\end{proof}

The following corollary extends and complements results from
\cite{AAK,BdPS}. It also relates to Question \ref{q3} and complements
results in \cite{BCS1}. We refer for the definition and description of the algebra
$LS(\mathcal{M})$ of all locally measurable operators affiliated to $\mathcal{M}$
used in the proof below to \cite{BdPS} and references therein.

\begin{corollary}\label{th2onetype} Let $\mathcal{M}$ be a type $I$ von Neumann algebra and
let $\mathbb{E}$ be a Banach bimodule of $\tau$-compact operators.
Any derivation $\delta\colon\mathcal{M}\to \mathbb{E}$ is inner.
\end{corollary}
\begin{proof}
Since the set of all linear combinations of projections in
$\mathcal{Z}(\mathcal{M})$ is dense in $\mathcal{Z}(\mathcal{M})$
(with respect to the uniform norm), we infer that from Theorem \ref{th2} that
$\delta|_{\mathcal{Z}(\mathcal{M})}=0$. By \cite[Theorem 4.8]{BCS2},
there exists a derivation $\delta': LS(\mathcal{M})\longrightarrow LS(\mathcal{M})$
such that $\delta'|_{\mathcal{M}}=\delta$. Hence,
$\delta'|_{\mathcal{Z}(\mathcal{M})}=0$ and by \cite[Theorem 2.1]{AAK}
the derivation $\delta'$ is inner. The assertion now follows from \cite[Corollary 8]{BSWA}.
\end{proof}

Since any commutative von Neumann algebra is a type $I$ von Neumann algebra, we obtain

\begin{corollary}\label{bcom}
Suppose that $\mathcal{M}$ be a commutative von Neumann algebra
and $\mathbb{E}$ is a Banach bimodule of $\tau$-compact operators.
Any derivation $\delta\colon\mathcal{M}\to \mathbb{E}$ vanishes.
\end{corollary}

Our last main result in this note should be compared with some results from \cite{KW} (see e.g. \cite[Theorem 14, Corollary 15]{KW}). In that paper, derivations from a von Neumann subalgebra $\mathcal{N}$ of $\mathcal{M}$ into $L_p$-ideals of $\tau$-compact operators in $\mathcal{M}$ were studied. Our present approach allows us to treat derivations $\delta: \mathcal{N}\to E(\tau)$ under an additional assumption of their continuity (Theorem \ref{th3}) and arbitrary derivations $\delta: \mathcal{M}\to E(\tau)$ (Corollary \ref{last}). We believe that our techniques is of independent interest, since it does not depend on special geometrical properties of (noncommutative) $L_p$-spaces underlying the proofs in \cite{KW} and allows us to treat a rather general class of Banach ${\mathcal M}$-bimodules of (unbounded) $\tau$-measurable operators.

Let $E(0,\infty)$ be a symmetric function space. Firstly, we observe that if  ${\mathcal M}$ is a semifinite von Neumann algebra acting  in a separable Hilbert space $H$, then the symmetric space $E(\mathcal{M},\tau)$ is separable if and only if $E(0,\infty)$ is separable (see \cite {Su1, Me} and \cite[Propositions 1.1, 1.2]{Su2}). In this case, the dual space $F:=E(0,\infty)^*$ is naturally identified with a symmetric space on $(0,\infty)$ \cite{KPS} and we have $F(\mathcal{M},\tau)=E(\mathcal{M},\tau)^*$ (see \cite[Proposition 2.8]{Y} and also \cite{DDP1}). Below, $U(\mathcal{M})$ denotes the group of all unitary elements in $\mathcal{M}$.

\begin{theorem}\label{th3}  Let ${\mathcal M}$ be a semifinite von Neumann algebra equipped with a faithful normal semifinite trace $\tau$ acting in a Hilbert space $H$ and let
$E(\tau)$ be a symmetric operator space, which is dual (with respect to the duality given by $(x,y):=\tau(xy)$) to a symmetric space $F(\tau)$
on $\mathcal{M}$. Suppose that one of the following assumptions holds.

(i). The spaces $H$, $E(\tau)$ and $F(\tau)$ are separable Banach spaces;

(ii). The Banach space $E(\tau)$ is reflexive.

Then any continuous derivation $\delta:
\mathcal{N} \longrightarrow E(\tau)$ from an arbitrary
von Neumann subalgebra $\mathcal{N}$ in $\mathcal{M}$ is inner, that is $\delta(x)=[d,x]$ for some $d\in
E(\tau)$. In addition, the element $d$ can be chosen to satisfy $\|d\|_E\leq\|\delta\|:=\|\delta\|_{\mathcal{E}\to \mathcal{E} }$.

\end{theorem}

\begin{proof} We shall present the proof for the special case when $ \mathcal{N}=\mathcal{M}$, the proof of the general case is exactly the same.
Let $x\in E(\tau),\ u\in U(\mathcal{M})$. We set
$$T_u(x):=uxu^*+\delta(u)u^*.$$
Recall that
\begin{equation}\label{invariance}
 \mu(x)=\mu (uxu^*),\quad \forall x\in E(\tau), \quad \forall u\in U(\mathcal{M}).
\end{equation}
Since $\delta(u)\in E(\tau)$ (and hence, obviously, $\delta(u)u^*\in E(\tau)$), we obtain from \eqref{invariance} that $T_u(x)\in E(\tau)$ for every $x\in E(\tau)$. The mapping $T_{u}$ is isometrical for any $u\in U(\mathcal{M})$, since
$\|T_u(x)-T_u(y)\|_{E(\tau)}=\|u(x-y)u^*\|_{E(\tau)}=\|x-y\|_{E(\tau)}$. We shall now prove that the mapping $u\to T_u$ is a group homomorphism sending $U(\mathcal{M})$ into the subgroup of all isometries of the Banach space $E(\tau)$. To this end, we have
$T_{u}T_{v}(x)=T_{u}(vxv^*+\delta(v)v^*)=
u(vxv^*+\delta(v)v^*)u^*+\delta(u)u^*=
(uv)^*x(uv)+u\delta(v)v^*u^*+\delta(u)vv^*u^*=
(uv)^*x(uv)+\delta(uv)(uv)^*=T_{uv}(x)$.  Finally, we note that $T_{u}$ is an affine mapping on  $E(\tau)$. Indeed, for any scalar $\alpha$, we have
$T_{u}(\alpha x+(1-\alpha)y)=u(\alpha
x+(1-\alpha)y)u^*+\delta(u)u^*= u(\alpha
x+(1-\alpha)y)u^*+\alpha\delta(u)u^*+(1-\alpha)\delta(u)u^*=\alpha
T_u(x)+(1-\alpha)T_u(y)$.

In other words, we have just defined the action $T$ of the group $U(\mathcal{M})$ on the Banach space
$E(\tau)$ by affine isometries. Consider the set $T(x):=\{T_u(x):\
u\in U(\mathcal{M})\}$, the orbit of an element  $x\in E$ under the action
$T$. We intend to consider the reduction of the action $T$ on the orbit $T(x)$ and on the weak or weak${}^*$ closure of its convex hull. (Here, we use the abbreviation weak${}^*$ closure to denote
the closure in the $\sigma(E(\tau), F(\tau))$-topology).
It follows immediately from the construction and the definition that each orbit is invariant under the action $T$. Since $T_{u}$ is an affine mapping for every
$u\in U(\mathcal{M})$, the convex hull $\rm{co}(T(x))$ of every orbit $T(x)$ is also invariant under the action $T$. Now, take $x=0$ and observe that
$$\|T_u(0)\|_{E(\tau)}=\|\delta(u)u^*\|_{E(\tau)}\leq\|\delta\|.$$
Thus, the set $T(0)$ and its convex hull $\rm{co}(T(0))$ are contained in $\|\delta\| E_1$, where $E_1$ is the unit ball of $E(\tau)$. Suppose that the assumption (i) holds.
By Banach-Alaoglu theorem, the set $K$, the weak${}^*$ closure of the set $\rm{co}(T(0))$,  is also a bounded weak${}^*$ compact subset of $E(\tau)$.
It is easy to check that $T_u(K)\subset K$ for any $u\in U(\mathcal{M})$. Indeed, to this end we need only to check that $\tau(yu^*x_\alpha u)\to \tau(yu^*xu)$ whenever $\tau(yx_\alpha )\to \tau(yx)$ for $\{x_\alpha\}, x\in E(\tau)$ and any $y\in F(\tau)$. However, the latter is immediate since $uyu^*\in F(\tau)$.
We are now in a position to apply the Namioka-Phelps fixed point theorem (see \cite[Theorem 15]{NP} and \cite[Corollary 10]{NP}) which guarantees that there exists an element $d\in K$ such that
$d=T_u(d)=udu^*+\delta(u)u^*$ for all $u\in U(\mathcal{M})$, in particular, $\delta(u)=du-ud=[d,u]$
for any $u\in U(\mathcal{M})$. Suppose that the assumption (ii) holds. In this case, the weak${}^*$ topology on $E(\tau)$ coincides with the weak topology and we infer the existence of an element $d$ with the same property as above thanks to the Ryll-Nardzewski  fixed point theorem (see e.g. \cite{SZ}). Observe that in this case, we do not need a separability assumption imposed on $E(\tau)$ and $F(\tau)$ in (i) in order to justify the application of the Namioka-Phelps fixed point theorem. Since any element
$x\in\mathcal{M}$ is a linear combination of unitaries,  we conclude
$\delta(x)=[d,x]$ for any $x\in E(\tau)$.
\end{proof}

The following corollary follows immediately from Theorem \ref{th3} and Corollary \ref{SS2}.

\begin{corollary}\label{last}
Let $\mathcal{M}$ be a semifinite von Neumann algebra and
$E(\tau)$ be as in Theorem \ref{th3}. Any derivation $\delta:
\mathcal{M} \longrightarrow E(\tau)$ is inner.
\end{corollary}

We complete by observing that the assumption (ii) of Theorem \ref{th3} is automatically satisfied when the symmetric function spaces $E(0,\infty)$ is reflexive \cite{DDP1}.

\begin{corollary}\label{Lp}
Let ${\mathcal M}$ be a semifinite von Neumann algebra equipped with a faithful normal semifinite trace $\tau$ and let $1<p<\infty$. Then any
derivation $\delta: \mathcal{M} \longrightarrow L_p(\mathcal{M},\tau)$ is inner.
\end{corollary}

\end{document}